\documentclass[twoside,leqno,twocolumn]{article}
\usepackage{ltexpprt}
\usepackage{sodafixes}

\usepackage[utf8]{inputenc}
\usepackage{amssymb}
\usepackage{authblk}
\usepackage{enumerate}
\usepackage{mathtools}

\newcommand{\R}{\mathbb{R}}
\newcommand{\Q}{\mathbb{Q}}
\newcommand{\Z}{\mathbb{Z}}

\newcommand{\opt}{{\min}} 
\def\ve#1{{\mathchoice{\mbox{\boldmath$\displaystyle\bf#1$}}
{\mbox{\boldmath$\textstyle\bf#1$}}
{\mbox{\boldmath$\scriptstyle\bf#1$}}
{\mbox{\boldmath$\scriptscriptstyle\bf#1$}}}}

\DeclareMathOperator    \intr      {int}
\DeclareMathOperator    \rec       {rec}
\DeclareMathOperator    \sign      {sign}

\title{An FPTAS for Minimizing Indefinite Quadratic Forms over Integers in Polyhedra}

\author[1]{Robert Hildebrand\thanks{robert.hildebrand@ifor.math.ethz.ch}}
\author[1]{Robert Weismantel\thanks{robert.weismantel@ifor.math.ethz.ch}}
\author[1]{Kevin Zemmer\thanks{kevin.zemmer@ifor.math.ethz.ch}}
\affil[1]{Department of Mathematics, IFOR, ETH Zürich}

\date{October 10, 2015}

\begin{document} 
\maketitle
\begin{abstract}
We present a generic approach that allows us to develop a fully po\-ly\-no\-mial-time approximation scheme (FTPAS) for minimizing nonlinear functions over the integer points in a rational polyhedron in fixed dimension. The approach combines the subdivision strategy of Papadimitriou and Yannakakis \cite{papadimitriou_approximability_2000} with ideas similar to those commonly used to derive real algebraic certificates of positivity for polynomials. Our general approach is widely applicable. We apply it, for instance, to the Motzkin polynomial and to indefinite quadratic forms $\ve x^T Q \ve x$ in a fixed number of variables, where $Q$ has at most one positive, or at most one negative eigenvalue. In dimension three, this leads to an FPTAS for general $Q$.
\end{abstract}

\section{Introduction}

Consider the problem
\begin{equation}
\label{eq:problem-general}
\min \{ f(\ve x) : \ve x \in P \cap \Z^n\}
\end{equation}
for $f \colon \R^n \to \R$, $P = \{ \ve x \in \R^n \,:\, A \ve x \le \ve b \}$,  $A \in \Z^{m \times n}$, and $\ve b \in \Z^m$.
We use the words \emph{size} and \emph{binary encoding length} synonymously.
The size of $P$ is the sum of the sizes of $A$ and $\ve b$.
We say that Problem~\eqref{eq:problem-general} can be solved in polynomial time if in time bounded by a polynomial in the size of its input, we can either determine that the problem is infeasible, find a feasible minimizer, or show that the problem is unbounded by exhibiting a feasible point $\ve{\bar x}$ and an integer ray $\ve{\bar r} \in \rec(P) := \{ \ve x\in \R^n : A \ve x \leq 0\}$ such that $f(\ve{\bar x} + \lambda \ve{\bar r}) \to -\infty$ as $\lambda \to \infty$.

The main focus of this paper is Problem~\eqref{eq:problem-general} with $f(\ve x) = \ve x^T Q \ve x$, where $Q \in \Z^{n \times n}$ is a symmetric matrix.
Note that if $Q$ is not symmetric, then we can replace it by $Q' = \frac{1}{2} Q + \frac{1}{2} Q^T$, which is symmetric and satisfies $\ve x^T Q \ve x = \ve x^T Q' \ve x$.
The input size of Problem~\eqref{eq:problem-general} with $f(\ve x) = \ve x^T Q \ve x$ is the sum of the sizes of $P$ and $Q$.
When $Q$ is positive semi-definite, $f(\ve x)$ is convex, whereas it is concave when $Q$ is negative semi-definite. In the former case, Problem~\eqref{eq:problem-general} with fixed $n$ and bounded $P$ can be solved in polynomial time by \cite{khackiyan_integer_2000}, whereas in the latter case by \cite{cook-hartmann-kannan-mcdiarmid-1992} (cf.\ Theorem~\ref{thm:convex-concave} below).

The computational complexity of mixed-integer po\-ly\-nomial optimization in fixed dimension was surveyed in \cite{loera_integer_2006, loera_mixed_integer_2008}, and they develop an FPTAS for maximizing non-negative polynomials over integer and mixed-integer points in a polytope, respectively.
They also use a weaker notion of approximation algorithm, that has also been used in~\cite{vavasis_approximation_1992},  called  \emph{fully polynomial-time weak-approximation scheme}.  It allows the accuracy to depend on the maximum and minimum values of $f$ on the feasible set.  That is, let $f_{\max}$ and $f_{\min}$ be the maximum and minimum values of $f$ on the feasible region.  Then a weak $\epsilon$-approximation $\ve x_\epsilon$ of the minimization problem satisfies
$$
| f(\ve x_\epsilon) - f_{\min} | \leq \epsilon (f_{\max} - f_{\min}). 
$$

When $n \le 2$ and $f$ is a polynomial of degree $d$, it has been shown in \cite{delpia_integer_2013} and \cite{delpia_minimizing_2015} that Problem~\eqref{eq:problem-general} can be solved in polynomial time for $d \le 3$ and in \cite{loera_integer_2006} that it is NP-hard for $d = 4$.
To the best of our knowledge, the complexity of Problem~\eqref{eq:problem-general} for fixed $n \ge 3$ and when $f$ is a polynomial of degree $d=2,3$ is still unknown.
 
We present an FPTAS for Problem~\eqref{eq:problem-general} with $f(\ve x) = \ve x^T Q \ve x$ and fixed $n$ in two cases: For $Q$ with at most one negative eigenvalue, and for $Q$ with at most one positive eigenvalue.
The numbers of positive, negative, and zero eigenvalues of $Q$ from a triple known as the \emph{inertia} of $Q$ and can be computed in  polynomial time (see \cite{bunch_partial_1974} or \cite{hartung_computational_2007}).

By \cite{delpia_mixed_2014}, for fixed $n$, it is possible to decide in polynomial time if Problem~\eqref{eq:problem-general} with $f(\ve x) = \ve x^T Q \ve x$ is bounded or not. If it is unbounded, \cite{delpia_mixed_2014} gives a certificate of unboundedness, whereas if it is bounded, it returns an upper bound $\tau$ on the size of an optimal solution that is polynomial in the input size.
Thus we can replace $P$ by $P \cap \{ \ve x \in \R^n \,:\, \| \ve x \|_\infty \le 2^{\tau} \}$.
Henceforth we assume that $P$ is bounded.

We use a notion of approximation that is common in combinatorial optimization and is akin to the maximization version used in~\cite{loera_integer_2006, loera_mixed_integer_2008} for maximizing non-negative polynomials over polyhedra, except here, we extend the notion to allow for the objective to take negative values.

\begin{Definition}
Let $\ve x_\opt$ be an optimal solution to Problem~\eqref{eq:problem-general} and let $\epsilon > 0$.
We say that $\ve x_\epsilon$ is an \emph{$\epsilon$-approximate solution} to Problem~\eqref{eq:problem-general} if $\ve x_\epsilon$ is feasible and one of the following hold:
\begin{enumerate}
\item $f(\ve x_\opt) > 0$ and $f(\ve x_\epsilon) \leq (1+\epsilon) f(\ve x_\opt)$,
\item $f(\ve x_\opt) < 0$ and $f(\ve x_\epsilon) \leq \frac{1}{1 + \epsilon} f(\ve x_\opt)$, 
\item $f(\ve x_\opt) = 0$ and $f(\ve x_\epsilon) = f(\ve x_\opt)$.
\end{enumerate}

We say an algorithm for Problem~\eqref{eq:problem-general} is a \emph{fully polynomial-time approximation scheme} if for any $\epsilon > 0$, in polynomial time in $\tfrac{1}{\epsilon}$ and the size of the input, the algorithm correctly determines whether the problem is feasible, and if it is, outputs an $\epsilon$-approximate solution $\ve x_\epsilon$.
\end{Definition}

In order to develop an FPTAS for classes of nonlinear functions to be
minimized over integer points in polyhedra, we propose a framework
that combines the techniques of Papadimitriou and Yannakakis~\cite{papadimitriou_approximability_2000} with ideas similar to those commonly used to derive certificates of positivity for polynomials over semialgebraic sets.
Generally speaking, in the latter context one is given a finite
number of ``basic polynomials'' $f_1,\ldots,f_m$ which are known to be
positive over the integers in a polyhedron $P$. A sufficient condition
to prove that another polynomial $f$ is positive over $P \cap \Z^n$ is
to find a decomposition of $f$ as a sum of products of a \emph{sum of squares} (SOS) polynomial
and a basic function $f_i$. A polynomial $p(\ve x)$ is SOS if there exist polynomials $q_1(\ve x), \ldots, q_m(\ve x)$  such that $p(\ve x) = \sum_{i = 1}^m q_i^2(\ve x)$.

We would like to use a similar approach to arrive at an FPTAS. 
Again we work with classes of ``basic functions''.  Then, for a given $f$, we try to detect a  decomposition of $f$ as a
finite sum of products of a so-called ``sliceable function'' and a basic
function $f_i$. Roughly speaking, sliceable functions --- thanks to the 
result of~\cite{papadimitriou_approximability_2000} --- can be approximated by subdividing the given
polyhedron. 

For instance, the  set
of all convex functions presented by a first order oracle that are nonnegative over $P \cap \Z^n$ could
serve as a class of basic functions, because we can solve Problem~\eqref{eq:problem-general} for any
member in the class in polynomial time when $n$ is
fixed.
The nonnegativity assumption implies sign-compatibility, which is a necessary property of the set of basic functions that will be introduced in Section~\ref{sec:sliceable}.
Another example is the set of all concave functions presented by an evaluation oracle that are
nonnegative over $P \cap \Z^n$. The same property holds true in this case.
These two examples demonstrate that we consider not only polynomials
$f_i$, but also more general classes of basic functions. In fact, this is key
to tackle the quadratic optimization problem.  For example, we can decompose the polynomial $x^2 + y^2 - z^2$ as the product of two non-polynomial functions: a basic function $\sqrt{x^2 + y^2} - |z|$ and a sliceable function $\sqrt{x^2 + y^2} + |z|$.  Our technique also applies, for instance, to the Motzkin polynomial, but to functions $f$ that are not polynomials as well (see Theorem~\ref{thm:sliceable-product-sum}).

We postpone the discussion of what we mean by basic and sliceable
functions to Section~\ref{sec:sliceable}.
As a consequence of our technique we easily derive the following result
that in the optimization community was an open question for quite several years.

\begin{theorem}
\label{thm:main}
Let $Q \in \Z^{n \times n}$ be a symmetric matrix and let $n$ be fixed. Then there is an FPTAS for Problem~\eqref{eq:problem-general} with $f(\ve x) = \ve x^T Q \ve x$ in the following cases:
\begin{enumerate}[(i)]
\item $Q$ has at most one negative eigenvalue;
\label{itm:main_min}
\item $Q$ has at most one positive eigenvalue.
\label{itm:main_max}
\end{enumerate}
\end{theorem}
We present a proof for Theorem~\ref{thm:main} in  Section~\ref{sec:quadratics}.
Note that Theorem~\ref{thm:main} (\ref{itm:main_max}) is equivalent to maximizing $\ve x^T Q \ve x$ with at most one negative eigenvalue.
Together with known results on quasi-convex and quasi-concave optimization, Theorem~\ref{thm:main} implies that for $n = 3$ there is an FPTAS for general $Q$.

\begin{corollary}
\label{cor:dim3}
Let $Q \in \Z^{n \times n}$ be a symmetric matrix and $n = 3$. Then there is an FPTAS for Problem~\eqref{eq:problem-general} with $f(\ve x) = \ve x^T Q \ve x$.
\end{corollary}

In order to prove this corollary, we need the following known results and notation, which are also used for our other results.
Given a convex set $C$, a function $f$ is \emph{quasi-convex} on $C$ if the set $\{ \ve x \in C \,:\, f(\ve x) \le \alpha \}$ is convex for all $\alpha \in \R$. Convex functions are quasi-convex on any convex set in their domain. We say that $f$ is \emph{quasi-concave} if $-f$ is quasi-convex.
A \emph{semi-algebraic set} in $\R^n$ is a subset of the form 
$\bigcup_{i=1}^s \bigcap_{j=1}^{r_i} \left\{ \ve x \in \R^n \,\, | \,\, f_{i,j}(\ve x)\ *_{i,j}\ 0 \right\}$
where $f_{i,j} \colon \R^n \to \R$ is a polynomial in $n$ variables and $*_{i,j}$ is either $<$ or $=$ for $i = 1, \dots, s$ and $j = 1, \dots, r_i$ (cf.\ \cite{bochnak_real_1998}). 
For a polyhedron $P$, denote by $P_I$ the \emph{integer hull} of $P$, i.e., the convex hull of $P \cap \Z^n$.

\begin{theorem}[\cite{khackiyan_integer_2000}, \cite{cook-hartmann-kannan-mcdiarmid-1992}]
\label{thm:convex-concave}
For $A \in \Z^{m \times n}$ and $\ve b \in \Z^m$, let $P = \{ \ve x \in \R^n \,:\, A \ve x \le \ve b \}$ be a bounded polytope, and let $n$ be fixed.
\begin{enumerate}[(i)]
\item If $C$ is a convex, semi-algebraic set given by polynomial inequalities of degree at most $d$ and with integral coefficients of size at most $l$, and $f$ is a polynomial which is quasi-convex on $P$, then $\min \{ f(\ve x) : \ve x \in P \cap C \cap \Z^n\}$ can be solved in polynomial time in $d$, $l$ and the size of $P$ and the size of the coefficients of all involved polynomials.%
\label{itm:quasi-convex-semi-algebraic}
\item The vertices of $P_I$ can be enumerated in polynomial time in the size of $P$.%
\label{itm:vertices-integer-hull}
\item If a comparison oracle for $f$ is available, i.e., an oracle that given $\ve x, \ve y \in \Z^n$ decides whether $f(\ve x) < f(\ve y)$ holds or not, and $f$ is quasi-concave on $P_I$, then Problem~\eqref{eq:problem-general} can be solved in polynomial time in the size of $P$.%
\label{itm:quasi-concave}
\end{enumerate}
\end{theorem}

\begin{proof}[Proof of Corollary~\ref{cor:dim3}]
The inertia of $Q$ can be computed in polynomial time (see \cite{bunch_partial_1974} or \cite{hartung_computational_2007}).
If $Q$ has no negative eigenvalues, then $f$ is convex because it is the sum of convex functions, so we can apply Theorem~\ref{thm:convex-concave} (\ref{itm:quasi-convex-semi-algebraic}) with $C = P$.
If $Q$ has no positive eigenvalues, then $f$ is concave, so we can can apply Theorem~\ref{thm:convex-concave} (\ref{itm:quasi-concave}).
If $Q$ has one negative eigenvalue, we can apply Theorem~\ref{thm:main} (\ref{itm:main_min}), whereas if there are two, we can apply Theorem~\ref{thm:main} (\ref{itm:main_max}).
\end{proof}

The remainder of this paper is outlined as follows:
In Section~\ref{sec:sliceable} we explain basic and sliceable functions and present the main tool of this paper.
In Section~\ref{sec:quadratics} we give some important properties of quadratics and prove Theorem~\ref{thm:main}.

\section{Basic and Sliceable Functions}
\label{sec:sliceable}
Let $\mathcal P$ be a set of polytopes that is closed under taking subsets, i.e., such that for all $\bar P \subseteq P$, $P \in \mathcal P$ implies $\bar P \in \mathcal P$.  
A set of \emph{basic functions} $\Omega_{\mathcal P}$ is a class of functions with a common encoding structure for all elements and the following properties:
\begin{enumerate}[(i)]
\item $\Omega_{\mathcal P}$ is  closed under addition and multiplication by $\lambda \ge 0$,
\item each $f \in \Omega_{\mathcal P}$ admits a FPTAS for Problem~\eqref{eq:problem-general} over each polytope $P \in \mathcal P$,
\item all functions in $\Omega_{\mathcal P}$ are sign compatible, i.e., $\sign(f(\ve x)) = \sign(g(\ve x))$ for all $f, g \in \Omega_{\mathcal P}$, $\ve x \in P$, $P \in \mathcal P$.
\end{enumerate}

Examples of classes of functions $\Omega_{\mathcal P}$ that satisfy the properties above are the set of sign-compatible convex function over $\mathcal P$ presented by a first order oracle (cf.~\cite{Grotschel1988, Oertel2014}), sign-compatible concave functions over $\mathcal P$ presented by an evaluation oracle (cf.~\cite{cook-hartmann-kannan-mcdiarmid-1992}), and the set of polynomials that are negative on some polytope over $\mathcal P$ (cf.~\cite{loera_integer_2006, loera_mixed_integer_2008}).

We next introduce a class of functions that can be combined with basic functions and preserve the property that they admit a FPTAS for Problem~\eqref{eq:problem-general}.   For this, we need the following notation.
 
For $\ve k \in \Z^l$, $\ve c^0 \in \Z^l$, $\ve c^1, \ldots, \ve c^l \in \Z^n$, and $\epsilon > 0$, define 
\begin{multline*}
B_{\ve k}(\ve c^0, \ve c^1, \ldots, \ve c^l, \epsilon) := \Big\{
\ve x \in \R^n
\ :\  \\
\begin{cases}
(1 + \epsilon)^{k_j-1} \le L_j(\ve x) \le (1 + \epsilon)^{k_j}
& \textrm{if } k_j \ge 1 \\
L_j(\ve x) = 0
& \textrm{if } k_j = 0 \\
-(1 + \epsilon)^{-k_j-1} \le L_j(\ve x) \le -(1 + \epsilon)^{-k_j}
& \textrm{if } k_j \le -1
\end{cases} \\
\textrm{for } j = 1, \ldots, l
\Big\}.
\end{multline*}
where $L_j(\ve x) := \ve c^j \cdot \ve x + c^0_j$.
Let $B(\ve c^0, \ve c^1, \ldots, \ve c^l, \epsilon) := \bigcup_{\ve k \in \Z^l} B_{\ve k}(\ve c^0, \ve c^1, \ldots, \ve c^l, \epsilon)$.
Note that for all $\ve c^0 \in \Z^l$, $\ve c^1, \ldots, \ve c^l \in \Z^n$, and $\epsilon > 0$ we have that $\Z^n \subseteq B(\ve c^0, \ve c^1, \ldots, \ve c^l, \epsilon)$.
However, $\R^n \not \subseteq B(\ve c^0, \ve c^1, \ldots, \ve c^l, \epsilon)$, because $B_{\ve k}(\ve c^0, \ve c^1, \ldots, \ve c^l, \epsilon)$ is not full-dimensional if any entry of $\ve k$ is equal to zero.

\begin{Definition}
\label{def:sliceable}
Let $C \subseteq \R^n$.
A function $s\colon C \to \R_{\ge 0}$ is called \emph{sliceable} over $C$ if there exist $\ve c^0 \in \Z^l$, $\ve c^1, \ldots, \ve c^l \in \Z^n$ and $\zeta \in \Z_{\ge 1}$ such that for all $\epsilon > 0$ and for all $\ve k \in \Z^l$, we have that 
\begin{equation}
\label{eq:sliceable}
s(\ve x) \le (1 + \epsilon) s(\ve y)
\end{equation}
for all $\ve x, \ve y \in C \cap B_{\ve k}(\ve c^0, \ve c^1, \ldots, \ve c^l, \epsilon / \zeta) \cap \Z^n$.

We define the size of a sliceable function $s$ as the sum of the sizes of $\ve c^0, \ve c^1, \ldots, \ve c^l$ and the unary encoding size of $\zeta$.
\end{Definition}

It follows from Definition~\ref{def:sliceable} that if a function $s$ is sliceable over $C$ with parameters $\ve c^0, \ve c^1, \ldots, \ve c^l, \zeta$, and $s(\bar{\ve x}) = 0$ for some $\bar {\ve x} \in C \cap B_{\ve k}(\ve c^0, \ve c^1, \ldots, \ve c^l, \epsilon / \zeta) \cap \Z^n$, then $s(\ve x) = 0$ for all $\ve x \in C \cap B_{\ve k}(\ve c^0, \ve c^1, \ldots, \ve c^l, \epsilon / \zeta) \cap \Z^n$.
Regardless of this propagation property, there exist large classes of sliceable functions.
We present several examples of non-trivial sliceable functions below. To avoid the propagation of the zeros, we give choices of the parameters $\ve c^0, \ve c^1, \ldots, \ve c^l, \zeta$ such that the zeros of the function lay on intersections of hyperplanes defined by $L_j(\ve x) = 0$ for $j \in J$ for some $J \subseteq 1, \ldots, l$.

The set of sliceable functions is not closed under additive inverses, but is in fact a proper semifield as the following lemma shows.
\begin{lemma} 
\label{lem:sliceable-properties}
Let $s, r \colon C \to \R_{\ge 0}$ be sliceable functions and let $\lambda \in \R_{\ge 0}$.
Then $\lambda s$, $s + r$ and $s \, r$ are sliceable with size polynomial in the size of $s$ and $r$. 
Moreover, if $s > 0$, then $\frac{1}{s}$ is sliceable with size polynomial in the size of $s$.
\end{lemma}
\begin{proof}
Let $s$ be sliceable with parameters $\ve c^0$, $\ve c^1$, \ldots, $\ve c^{l_s}$, $\zeta_s$ and $r$ with parameters $\ve d^0$, $\ve d^1$, \ldots, $\ve d^{l_r}$, $\zeta_r$.  Set $\zeta_{s + r} = \max\{ \zeta_s, \zeta_r\}$ and $\zeta_{sr} = 4 \max\{\zeta_s, \zeta_r\}$.  We show that $s+r$ and $sr$ are sliceable with the same parameters $[\ve c^0; \ve d^0], \ve c^1, \ldots, \ve c^{l_s}, \ve d^1, \ldots, \ve d^{l_r}$ and  with $\zeta_{s+r}, \zeta_{sr}$, respectively.
Then for $\zeta = \zeta_{s + r}, \zeta_{s r}$ and $0 < \epsilon < 1$,
\begin{multline*}
B_{[\ve k^1; \ve k^2]}([\ve c^0; \ve d^0], \ve c^1, \ldots, \ve c^{l_s}, \ve d^1, \ldots, \ve d^{l_r}, \epsilon/ \zeta)
\\
\subseteq B_{\ve k^1}(\ve c^0, \ve c^1, \ldots, \ve c^{l_s},  \epsilon/\zeta ) \cap  B_{\ve k^2}( \ve d^0, \ve d^1, \ldots, \ve d^{l_r}, \epsilon/\zeta).
\end{multline*}

Let
$
\ve x, \ve y \in B_{[\ve k^1; \ve k^2]}\big([\ve c^0; \ve d^0], \ve c^1, \ldots, \ve c^{l_s}, \ve d^1, \ldots, \ve d^{l_r},
\allowbreak
\epsilon/ \zeta\big) \cap \Z^n.
$
By the containment, it follows that 
\begin{equation*}
s(\ve x) \leq (1 + \epsilon \tfrac{\zeta_s}{ \zeta}) s(\ve y) \leq (1 + \epsilon) s(\ve y)
\end{equation*}
and 
\begin{equation*}
r(\ve x)\leq (1 + \epsilon \tfrac{\zeta_r}{ \zeta }) r(\ve y) \leq (1 + \epsilon) r(\ve y) .
\end{equation*}

For $\zeta = \zeta_{s + r}$, by combining these, we see that $s(\ve x) + r(\ve x) \leq (1+\epsilon) (s(\ve y) + r(\ve y))$.  Hence, $s+r$ is sliceable.  
Instead, for $\zeta = \zeta_{s r}$, by combining these, we see that $sr$ is also sliceable since
\begin{align*}
s(\ve x) r(\ve x)
&\leq (1 + \epsilon \tfrac{\zeta_s}{ \zeta_{sr}}) s(\ve y) (1 + \epsilon \tfrac{\zeta_r}{ \zeta_{sr} }) r(\ve y)
\\
&\leq (1 + \tfrac{\epsilon}{4})^2 s(\ve y) r (\ve y)
\\
&\leq (1+ \epsilon) s(\ve y) r(\ve y).
\end{align*}

The function $\lambda s$ is trivially sliceable the same parameters as $s$.  Also, if $s > 0$, then it follows $\tfrac{1}{s}$ is sliceable with the same parameters as $s$.  
Indeed, for $\ve x, \ve y \in B_{\ve k}(\ve c^0, \ve c^1, \ldots, \ve c^{l_s},  \epsilon/\zeta_s ) \cap \Z^n$
it holds that  $s(\ve y) \leq (1+\epsilon) s(\ve x)$.  By dividing, we have $\tfrac{1}{s(\ve x)} \leq (1+ \epsilon) \tfrac{1}{s(\ve y)}$.
\end{proof}

Any function that is the sum of even powers of linear forms is sliceable over $\R^n$.  Rational powers of sliceable functions are also sliceable.  For instance,  $\sqrt{ s(\ve x)}$ is sliceable for any sliceable function $s$.
Moreover, any polynomial that is positive on a polyhedron $P$ is sliceable on $P$.  Handelman~\cite{handelman1988} showed that any polynomial $f$ that is positive over a polyhedron 
 has a representation as $f(\ve x) = \sum_{|\alpha| \leq L, \alpha \in \Z_+^n} \beta_\alpha \prod_{i=1}^n (\ve a^i \cdot \ve x - b_i)^{\alpha_i}$ for some $L \in \Z_+$ and $\beta_\alpha \geq 0$.  
It follows that $f$ is sliceable over $P$.  Unfortunately, the parameter $\zeta$ grows exponentially with $L$.  The best bounds on $L$ are given by Powers and Reznick~\cite{Powers2001}, which are in terms of the coefficients of $f$ and the smallest value that $f$ attains on $P$.  The smallest value is bounded by a polynomial in the input provided the degree and the dimension are fixed~\cite{Jeronimo2010}, but this still implies that the best known bound for $L$ is not of polynomial size in the input.  Therefore there is no known bound on the size of $\zeta$ based on the input of the polynomial $f$ that would admit an FPTAS for Problem~\eqref{eq:problem-general} over $P$ with this decomposition.

\begin{Example}{}
The polynomial $M(x, y) = x^4 y^2 + x^2 y^4 - 3 x^2 y^2 + 1$ is known as the Motzkin polynomial. While it is not SOS, it is known to be positive because $(x^2 + y^2) \cdot M(x, y)$ has an explicit SOS decomposition (see \cite[p.~84]{parrilo_polynomial_2012}).  In particular, for $x,y \neq 0$, we have  
\begin{align*}
M(x,y) 
&= \frac{1}{(x^2 + y^2)} \Big( y^2(1-x)^2(1+x)^2 
\\
&+ x^2(1-y)^2(1+y)^2 + x^2 y^2 (x^2 + y^2 - 2)^2 \Big).
\end{align*}
Thus, for $x,y \neq 0$, $M(x,y)$ is a composition of sliceable functions and a convex function.  That is, the parameters $\ve c^0, \ve c^1, \ldots, \ve c^l$ are defined by $c^j_1 x + c^j_2 y + c^0_j = F_j(x, y)$ for $F_1(x, y), \ldots, F_6(x, y)$ equal to $x,y, 1-x, 1+x, 1-y, 1+y$ respectively, and the convex function is $(x^2 + y^2 - 2)^2$.  Sign-compatibility is also given because all involved functions are non-negative.  As we will see in Theorem~\ref{thm:sliceable-product-sum}, this implies that Problem~\eqref{eq:problem-general} with $f(x, y) = M(x,y)$ admits an FPTAS over polytopes.
\end{Example}

The following lemma allows us to find approximate solutions by taking the best approximate solution over a decomposition of the feasible region.

\begin{lemma}
\label{lem:main-tool}

Let $C \subseteq \R^n$ such that $C \cap \Z^n \ne \emptyset$, and for $j=1, \dots, m$ let $s_j \colon C \to \R_{\geq 0}$ for $j=1, \dots, m$ 
and let $g_j \colon C \to \R$.
Let $0 < \epsilon < 1$ and $\epsilon' = \epsilon / 4$. Let $K \subseteq \Z^n$ such that $C \cap \Z^n \subseteq \bigcup_{\ve k \in K} B_{\ve k}$, where $B_{\ve k} \subseteq \R^n$  are polytopes.
Also, let $L_j^{\ve k} \geq 0$ such that $L_j^{\ve k} \leq s_j(\ve x) \leq (1+\epsilon') L_j^{\ve k}$ for all $\ve x \in B_{\ve k} \cap \Z^n$ and $j=1, \dots, m$.

Then, for
\begin{equation*}
\ve x_\epsilon := \arg \min \left\{ \sum_{j=1}^m g_j(\ve{ x}^{\ve k}_{\epsilon'}) s_j(\ve{ x}^{\ve k}_{\epsilon'}) : \ve k \in K \right\}
\end{equation*}
where  $\ve{ x}^{\ve k}_{\epsilon'}$ is an $\epsilon'$-approximate solution to 
\begin{equation*}
\min \left\{ \sum_{j=1}^m g_j(\ve x) L_j^{\ve k} : \ve x \in C \cap B_{\ve k} \cap \Z^n \right\},
\end{equation*}
we have that $\ve x_\epsilon$ is an $\epsilon$-approximate solution to 
\begin{equation*}
\min \left\{ \sum_{j=1}^m g_j(\ve x) s_j(\ve x) : \ve x \in C \cap \Z^n \right\} .
\end{equation*}
\end{lemma}
\begin{proof}
Let
\begin{equation*}
\ve x_\opt := \arg \min \left\{ \sum_{j=1}^m g_j(\ve x) s_j(\ve x) : \ve x \in C \cap \Z^n \right\} ,
\end{equation*}
let $\ve{\bar k} \in K$ such that $\ve x_\opt \in B_\ve{ \bar k}$, and define $f(\ve x) := \sum_{j=1}^m g_j(\ve x) s_j(\ve x)$.
Suppose first that $f(\ve{ x}^{\ve{\bar k}}_{\epsilon'}) \ge 0$, which implies that $g_j(\ve{ x}^{\ve{\bar k}}_{\epsilon'}) \ge 0$ for all $j = 1, \ldots, m$, because all functions in $\Omega_{\mathcal P}$ are sign compatible and $s_j \ge 0$. Then 
\begin{align*}
f(\ve{ x}^{\ve{\bar k}}_{\epsilon'})
&= \sum_{j=1}^m g_j(\ve{ x}^{\ve{\bar k}}_{\epsilon'}) s_j(\ve{ x}^{\ve{\bar k}}_{\epsilon'})
\\
&\le (1 + \epsilon') \sum_{j=1}^m g_j(\ve{ x}^{\ve{\bar k}}_{\epsilon'}) L_j^{\ve k}
\\
&\le (1 + \epsilon')^2 \sum_{j=1}^m g_j(\ve x_\opt) L_j^{\ve k}
\\
&\le (1 + \epsilon')^2 \sum_{j=1}^m g_j(\ve x_\opt) s_j(\ve x_\opt)
\\
&= (1 + \epsilon')^2 f(\ve x_\opt)
\le (1 + \epsilon) f(\ve x_\opt).
\end{align*}
The second inequality follows from approximate minimality of $\ve{x}^{\ve{\bar k}}_{\epsilon'}$, which at the same time implies that $g_j(\ve x_\opt) \ge 0$ for all $j = 1, \ldots, m$, as otherwise $f(\ve{ x}^{\ve{\bar k}}_{\epsilon'}) < 0$.
In particular, this implies that $f(\ve x_\opt) \ge 0$.

Suppose now that $f(\ve{ x}^{\ve{\bar k}}_{\epsilon'}) < 0$, which implies that $g_j(\ve{ x}^{\ve{\bar k}}_{\epsilon'}) \le 0$ for all $j = 1, \ldots, m$ because all functions in $\Omega_{\mathcal P}$ are sign compatible and $s_j \ge 0$. Then 
\begin{align*}
f(\ve{ x}^{\ve{\bar k}}_{\epsilon'})
&= \sum_{j=1}^m g_j(\ve{ x}^{\ve{\bar k}}_{\epsilon'}) s_j(\ve{ x}^{\ve{\bar k}}_{\epsilon'})
\le \sum_{j=1}^m g_j(\ve{ x}^{\ve{\bar k}}_{\epsilon'}) L_j^{\ve k}
\\
&\le \frac{1}{1 + \epsilon'} \sum_{j=1}^m g_j(\ve x_\opt) L_j^{\ve k}
\\
&\le \frac{1}{(1 + \epsilon')^2} \sum_{j=1}^m g_j(\ve x_\opt) s_j(\ve x_\opt)
\\
&= \frac{1}{(1 + \epsilon')^2} f(\ve x_\opt)
\le \frac{1}{1 + \epsilon} f(\ve x_\opt).
\end{align*}
To derive those inequalities, we have again used the assumptions on $L^{\ve k}_j$ and approximate minimality of $\ve{ x}^{\ve{\bar k}}_{\epsilon'}$.
In particular, the last inequality holds because $f(\ve x_\opt) \le 0$.
The result follows from $f(\ve x_\epsilon) \le f(\ve{ x}^{\ve{\bar k}}_{\epsilon'})$.
\end{proof}

Let $\mathcal S(\mathcal P)$ be the set of sliceable functions over all $P \in \mathcal P$ with an evaluation oracle.
We define 
\begin{equation*}
\mathcal C(\Omega_{\mathcal P}) := \left\{ \sum_{j=1}^m g_j s_j \,|\, g_j \in \Omega_{\mathcal P}, s_j \in \mathcal S(\mathcal P), m \textrm{ finite} \right\} .
\end{equation*}

We now show that there is an FPTAS for all functions in the class $\mathcal C (\Omega_{\mathcal P})$. 
\begin{theorem}
\label{thm:sliceable-product-sum}
Let $\mathcal P$, $\Omega_{\mathcal P}$ and $\mathcal C(\Omega_{\mathcal P})$ be as above.
Suppose that there is an evaluation oracle for $f = \sum_{j=1}^m g_j s_j \in \mathcal C(\Omega_{\mathcal P})$ where $g_j \in \Omega_{\mathcal P}$, $s_j \in \mathcal S(\mathcal P)$, $j=1, \dots, m$.  Then for every $P \in \mathcal P$, 
 there exists an FPTAS for Problem~\eqref{eq:problem-general}, where there size of the input is the sum of the sizes of $g_j$, $s_j$ for $j = 1, \ldots, m$, and $P$.
\end{theorem}
\begin{proof}
Let $f = \sum_{j=1}^m g_j s_j \in \mathcal C(\Omega_{\mathcal P})$, $P \in \mathcal P$, and let $\epsilon > 0$ . 
Feasibility of $P \cap \Z^n$ can be determined in polynomial time by \cite{lenstra_integer_1983}.
Thus we assume henceforth that $P \cap \Z^n \ne \emptyset$.
 
Let the slicing parameters of $s_j$ be $\ve c^0_j, \ve c^1_j, \ldots, \ve c^{l_j}_j, \zeta_j$ for $j=1, \dots,m$ and let $\zeta_{\max} = \max_j \{ \zeta_j\}$.
Since $P$ is bounded, there exists $R \in \Z_{\ge 1}$ of polynomial size in the size of $P$ and the sizes of $s_j$, $j=1, \dots, m$, such that $| \ve c_j^i \cdot \ve x + c^0_j | \leq R$ for all $\ve x \in P$, $i=1, \ldots, l_j$, and $j=1, \ldots, m$ (see, for example, \cite{schrijver_theory_1986}).   
Set $\epsilon' = \tfrac{\epsilon}{4}$ and $N := \left\lceil \frac{\log_2(R)}{\log_2 (1+\epsilon'/ \zeta_{\max})} \right\rceil \leq  \left\lceil \log_2(R) ( \tfrac{\zeta_{\max}}{\epsilon'} + 1) \right\rceil   $.   Define
\begin{align*}
B_{\ve k}
:= B_{\ve k}\big(&[\ve c^0_1, \ldots, \ve c^0_m],
\\
&\ve c^1_1, \ldots, \ve c^{l_1}_1, \ldots, \ve c^1_m, \ldots, \ve c^{l_m}_m, \epsilon'/ \zeta_{\max}\big).
\end{align*}
In particular, $\Z^n \subseteq \bigcup_{\ve k \in \Z^{l_1 + \ldots + l_m}} B_{\ve k}$. Then for $K =   \{-N, \ldots, N\}^{\sum_{j=1}^m l_j}$ we have $P \cap \Z^n \subseteq \bigcup_{\ve k \in K} B_{\ve k}$.  
Notice that many choices of $\ve k \in K$ are redundant since each non-empty set $B_{\ve k}$ is a cell or facet of a cell of the hyperplane arrangement in $\R^n$ given by the hyperplanes $\ve c^i_j \cdot \ve x + c^0_j = 0$ and $\ve c^i_j \cdot \ve x + c^0_j = \pm(1+\epsilon'/\zeta_{\max})^k$ for $k \in \{0, \ldots, N\}$, $i=1, \ldots, l_j$, and $j=1, \ldots, m$.   
By \cite{sleumer_output_1998}, we can enumerate the $O \left( (2N + 3)^n \left( \sum_{j=1}^m l_j\right)^n \right)$ cells of the hyperplane arrangement in polynomial time with $n$ fixed.  Hence, we can find a subset $K' \subseteq K$ with cardinality of polynomial size such that $P \subseteq \bigcup_{\ve k \in K'} B_{\ve k}$.

Let $\ve x_\epsilon$ be as in Lemma~\ref{lem:main-tool}.
For $\ve k \in K'$, we determine a point $\ve x^{\ve k} \in P \cap B_{\ve k} \cap \Z^n$ using~\cite{lenstra_integer_1983}.  If no such point exists, then we remove $\ve k$ from $K'$.  Set $L_j^{\ve k} := \frac{1}{1 + \epsilon'} s_j(\ve x^{\ve k})$, computed using the evaluation oracle for $s_j$, for $j=1, \ldots, m$.
Note that $L_j^{\ve k}$ satisfies the assumptions of Lemma~\ref{lem:main-tool} because $s_j$ is sliceable and $\zeta_j \leq \zeta_{\max}$.
Then, by Lemma~\ref{lem:main-tool}, $\ve x_\epsilon$ is an $\epsilon$-approximate solution to Problem~\eqref{eq:problem-general}.

We show that we can compute $\ve x_\epsilon$ in polynomial time.  
For each $\ve k \in K'$, an $\epsilon'$-approximate solution $\ve{x}^{\ve k}_{\epsilon'}$ to
\begin{equation*}
\min \left\{ \sum_{j=1}^m g_j(\ve x) L_j^{\ve k} : \ve x \in P \cap B_{\ve k} \cap \Z^n \right\}
\end{equation*}
can be computed in polynomial time because $\Omega_{\mathcal P}$ is closed under addition and multiplication by positive scalars, and therefore $\sum_{j=1}^m g_j(\ve x) L_j^{\ve k} \in \Omega_{\mathcal P}$.
Finally, $\arg \min \{ f(\ve{x}^{\ve k}_{\epsilon'}) : \ve k \in K' \}$ can be computed in polynomial time using the evaluation oracle of $f$ because $|K'|$ is polynomial in the input size.
\end{proof}

\begin{Remark}
\label{rem:main-m=1}
The assumptions of Theorem~\ref{thm:sliceable-product-sum} can be slightly relaxed if we only consider functions of the form $f(\ve x) = g_1(\ve x) s_1(\ve x)$, i.e., $\Omega_{\mathcal P} = \{ \lambda g_1 : \lambda \ge 0 \}$ for some $g_1 \colon \R^n \to \R$ and $s_1$ sliceable.
In fact, in this case no evaluation oracle for $s_1$ is needed, because an $\epsilon'$-approximate solution to 
$$
\min \left\{ g_1(\ve x) : \ve x \in P \cap B_{\ve k} \cap \Z^n \right\}
$$
is also an $\epsilon'$-approximate solution to 
$$
\min \left\{ g_1(\ve x) L^{\ve k}_1 : \ve x \in P \cap B_{\ve k} \cap \Z^n \right\}
$$
for any $L^{\ve k}_1 \ge 0$. Thus, $L^{\ve k}_1$ does not need to be computed explicitly.
\end{Remark}

\begin{Remark}
Theorem~\ref{thm:sliceable-product-sum} can be generalized from bounded polyhedra $P$ to more general sets $C$, provided that for any bounded polyhedron $\bar P$ the feasibility and minimization problems on $C \cap \bar P \cap \Z^n$ and the feasibility problem on $C \cap \Z^n$ can be solved in polynomial time in the size of $\bar P$ and the bit size required to represent $C$. The class of convex semi-algebraic sets in fixed dimension is an example of such a class, because for fixed dimension the stated feasibility problem can be solved in polynomial time by \cite{khackiyan_integer_2000}.
\end{Remark}

\begin{Remark}
The size of a sliceable function $s$ is defined as the sum of the sizes of its slicing parameters, where unary encoding is assumed for the parameter $\zeta$. With this definition, a polynomial bound on the cardinality of $K'$ in terms of the sizes of the sliceable functions is derived in the proof of Theorem~\ref{thm:sliceable-product-sum}.
This bound is not polynomial if binary encoding is assumed for the $\zeta$ parameters.
However, in our main application of  Theorem~\ref{thm:sliceable-product-sum}, namely the proof of Theorem~\ref{thm:main}, those parameters can be chosen to be constant.
\end{Remark}

\section{FPTAS for Quadratic Forms}
\label{sec:quadratics}

For positive definite matrices, it is well known that a Cholesky decomposition $Q = L D L^T$ can be easily computed, where $L$ is a lower-triangular matrix and $D$ is a diagonal matrix.  Such a decomposition does not always exist for indefinite symmetric matrices~\cite{dax_pivoting_1977}.
However, there is a similar decomposition that serves our purpose, as described in the following remark.
\begin{Remark}[\cite{dax_pivoting_1977, hartung_computational_2007}]
\label{rem:indefinite-decomposition}
For any symmetric matrix $Q \in \Z^{n \times n}$, we can find an invertible matrix $S \in \Q^{n \times n}$ and a diagonal matrix $D \in \Q^{n \times  n}$ such that $Q = S D S^T$ in polynomial time in the size of $Q$ for fixed $n$.
\end{Remark}
We thus have 
\begin{equation*}
\ve x^T Q \ve x = \sum_{i=1}^n d_i (S_i \cdot \ve x)^2
\end{equation*}
where $D = \mathrm{diag}(\ve d)$ is the diagonal matrix with $D_{ii} = d_i$, and $S_i$ denotes the $i$-th row of $S$.

\begin{Remark}
\label{rem:FPTAS-scaling}
Since an FPTAS for $c^3 f(\ve x)$, $c \in \Q_{> 0}$, is also an FPTAS for $f(\ve x)$ as long as $c$ is of polynomial size in the input size, $S$ and $D$ can be scaled to be in $\Z^{n \times n}$. In fact, let $c$ be the least common multiple of the denominators of all entries of $S$ and $D$, which is of polynomial size in the size of $Q$, because $S$ and $D$ can be computed in polynomial time in the size of $Q$. Then $c^3 f(\ve x) = \ve x^T (cS) (cD) (cS)^T \ve x$ and $cS, cD \in \Z^{n \times n}$.
\end{Remark}

Since $S$ is non-singular, Sylvester's law of inertia
(see, for example \cite{golub_matrix_1996} or \cite{gill_numerical_1991})
states that the inertia of $Q$ is the same as the inertia of $D$. Therefore $D$ has exactly as many positive, negative and zero diagonal entries as $Q$ has positive, negative and zero eigenvalues.

The following lemmas will be useful for determining the accuracy necessary for the numerical approximations.
\begin{lemma}[see, e.g.,~\cite{delpia_minimizing_2015}]
\label{thm:polynomial_root_bounds}
Let $m < n$ be nonnegative integers, $a_m, \ldots, a_n \in \R$, $a_m \ne 0$, $a_n \ne 0$, and let $\bar x \in \R$ be a nonzero root of the polynomial $f(x) := \sum_{i=m}^n a_i x^i$. Then 
$|\bar x| > \min \{ |a_m|/(|a_m| + |a_i|) \, : \, i = m+1, \ldots, n \}$.
\end{lemma}

The proof of the next lemma uses a standard technique for converting an equation with rational powers to one with integer powers.
\begin{lemma}
\label{lem:difference}
Let $p,q,p',q' \in \Z$ with $|p|, |q|, |p'|, |q'| \leq R$, $p,p' \geq 0$, and let $\delta = (\sqrt{p} + q) - (\sqrt{p'} + q')$.  If $\delta \neq 0$, then $|\delta| >  \tfrac{1}{5\cdot 48} R^{-4}$.
\end{lemma}
\begin{proof}
We show that any solution $\delta$ is a solution to a polynomial equation in terms of $\delta$.  Rewriting the equation for $\delta$ by solving for $\sqrt{p}$ and squaring both sides yields
$p = \delta^2 + p' + q^2 - 2 q q' + q'^2 + \delta (-2 q + 2 q') + \sqrt{p'} (2 \delta - 2 q + 2 q')$.
Solving for $\sqrt{p'}(2 \delta - 2 q + 2 q')$ and squaring both sides yields a degree 4 polynomial in $\delta$ with integer coefficients that are polynomial in $R$.    In particular, it can be shown the the absolute values of the coefficients are upper bounded by $48 R^4$.  Since $\delta$ is a solution to this polynomial equation, by Lemma~\ref{thm:polynomial_root_bounds}, it follows that if $\delta \neq 0$, then $|\delta| >  \tfrac{1}{5\cdot 48} R^{-4}$.
\end{proof}

With those technical tools, we can now prove Theorem~\ref{thm:main}.
\begin{proof}[Proof of Theorem~\ref{thm:main}]
Given $Q \in \Z^{n \times n}$, it is possible to compute $S$ and $D = \mathrm{diag}(\ve d)$ such that $Q = S D S^T$ in polynomial time by Remark~\ref{rem:indefinite-decomposition}.
Moreover, $S$ and $D$ can be assumed to be integral by Remark~\ref{rem:FPTAS-scaling}.
Suppose that $Q$ has at most one negative eigenvalue, which implies that $-Q$ has at most one positive eigenvalue.
By reordering variables, we may assume that $d_i \ge 0$ for $i = 1, \ldots, n-1$.
Note that if $d_n \ge 0$, then $f_1(\ve x) = \ve x^T Q \ve x$ is convex and $f_2(\ve x) = \ve x^T (-Q) \ve x$ is concave, so Problem~\eqref{eq:problem-general} with $f = f_1$ and $f = f_2$ can be solved in polynomial time by Theorem~\ref{thm:convex-concave} (\ref{itm:quasi-convex-semi-algebraic}) and (\ref{itm:quasi-concave}) respectively.
Thus assume henceforth that $d_n < 0$.

Let 
\begin{equation*}
g(\ve x) := \sqrt{\sum_{i=1}^{n-1} d_i (S_i \cdot \ve x)^2} -  \sqrt{- d_n} |S_n \cdot \ve x|
\end{equation*}
and
\begin{equation*}
s(\ve x) := \sqrt{\sum_{i=1}^{n-1} d_i (S_i \cdot \ve x)^2} +  \sqrt{- d_n} |S_n \cdot \ve x| .
\end{equation*}
Notice that $f_1(\ve x) = g(\ve x) s(\ve x)$ and $f_2(\ve x) = (-g(\ve x)) s(\ve x)$.  
Let $\mathcal P := \{ \bar P \text{ polytope } : \bar P \subseteq P\}$ and $\Omega_{\mathcal P} := \{ \lambda g : \lambda\in \R\}$.
We will show that $s$ is sliceable with size bounded by a polynomial in the size of $Q$, and that $\Omega_{\mathcal P}$ is a class of functions that we can minimize over the integers in polyhedra $\bar P \in \mathcal P$.  Equivalently, we show that we can minimize $g$ and $-g$ over $\bar P \cap \Z^n$ for $\bar P \in \mathcal P$.   The result then follows from Theorem~\ref{thm:main} and Remark~\ref{rem:main-m=1}.

First,  we define the slicing parameters $\ve c^j = S_j$ for $j = 1, \ldots, n$, $\ve c^0 = 0$, and $\zeta = 1$.  
Then for any $\ve k \in \Z^n$, $\epsilon > 0$, and any $\ve x, \ve y \in B_{\ve k}(\ve c^0, \ve c^1, \ldots, \ve c^n, \epsilon)$, 
we have
\begin{align*}
s(\ve x) 
&=  \sqrt{\sum_{i=1}^{n-1} d_i (S_i \cdot \ve x)^2} +  \sqrt{- d_n} |S_n \cdot \ve x|
\\
&\leq \sqrt{\sum_{i=1}^{n-1} d_i ((1+\epsilon) S_i \cdot \ve y)^2} +  \sqrt{- d_n} |(1+\epsilon) S_n \cdot \ve y|
\\
&\leq (1+\epsilon) s(\ve y) .
\end{align*}
Hence, $s$ is sliceable with size polynomial in the size of $Q$.

Second, we show that we can solve $\min\{g(\ve x) : \ve x \in \bar P \cap \Z^n\}$ exactly for any $\bar P \in \mathcal P$.
The function $g$ is convex on the halfspaces $S_n \cdot \ve x \ge 0$ and $S_n \cdot \ve x \le 0$, because it is the sum of a norm with a linear function.  Even though $g$ is convex, an exact first order oracle is not available because the function values can be irrational.  We show the minimization problem can be solved by reformulating sublevel sets  as convex semi-algebraic sets.  

We separately consider the problem on the halfspaces $S_n \cdot \ve x \ge 0$ and $S_n \cdot \ve x \le 0$.  The analysis for both cases is similar.  Without loss of generality, assume that $S_n \cdot \ve x \geq 0$.  

For $\beta \in \R$, the inequality  $g(\ve x) \le \tfrac{\beta}{\sqrt{-d_n}}$ is equivalent to
\begin{equation*}
\sqrt{|d_n| \sum_{i=1}^{n-1}  d_i (S_i \cdot \ve x)^2} \le \beta + |d_n| S_n \cdot \ve x.
\end{equation*}
This is a second order cone constraint, which is well known to define a convex set~\cite{Boyd2004}.
It can be reformulated by recognizing that the left hand side is nonnegative and adding the condition that the right hand side is nonnegative too. With this, we can square both sides to get the equivalent relation 
\begin{align}
\label{eq:convex-set}
\begin{split}
|d_n| \sum_{i=1}^{n-1}  d_i (S_i \cdot \ve x)^2 -  (\beta + |d_n| S_n \cdot \ve x)^2 &\le 0,
\\
\beta + |d_n| S_n \cdot \ve x &\ge 0 .
\end{split}
\end{align}
In this representation, the set $\{\ve x \in \R^n : g(\ve x) \leq \tfrac{\beta}{\sqrt{-d_n}}, \ S_n \cdot \ve x \geq 0\}$ is a convex semi-algebraic set. It is thus possible to test integer feasibility of this set for $\beta \in \Q$ in polynomial time using Theorem~\ref{thm:convex-concave} (\ref{itm:quasi-convex-semi-algebraic}).  
Since $\bar P$ is bounded, and $S$ and $D$ are of polynomial size, it follows from~\cite{schrijver_theory_1986} that we can compute $R \in \Z$, $R \ge \|\ve d\|_\infty$ of polynomial size such that $\| S \ve x \|_\infty \leq R$ for all $\ve x \in \bar P$. 
Note that $|\sqrt{-d_n} \, g(\ve x)| \leq 2nR^3$ for all $\ve x \in \bar P$.  
We now perform binary search on $\beta$ in the interval $[-2nR^3, 2nR^3]$ for integer feasibility of~\eqref{eq:convex-set} with precision $\mu := \tfrac{1}{5\cdot 48} n^{-4}R^{-16}$.  Since these numbers are of polynomial size, the binary search can be done in polynomial time.
We show now that the choice of precision yields an optimal solution by showing that it distinguishes function values of $\sqrt{-d_n}\, g(\ve x)$ for  $\ve x \in \bar P \cap\Z^n$.  

Indeed, consider $\ve x,\ve y \in \bar P \cap \Z^n$ and define $\delta := \sqrt{-d_n}( g(\ve x) - g(\ve y))$.  
Setting 
$p = |d_n| \sum_{i=1}^{n-1} d_i (S_i \cdot \ve x)^2$, $q =  - |d_n| S_n \cdot \ve x$, 
$p' = |d_n| \sum_{i=1}^{n-1} d_i (S_i \cdot \ve y)^2$, and $q' = -|d_n| S_n \cdot \ve y$, 
 we see that $\delta = (\sqrt{p} + q) - (\sqrt{p'} + q')$.  
Since $p, p' \geq 0$ and $|p|,|p'|,|q|,|q'| \leq n R^4$, by Lemma~\ref{lem:difference}, if $\delta \neq 0$, then $|\delta| > \mu$.
Thus our choice of the precision for the binary search distinguishes function values of $\sqrt{-d_n} g(\ve x)$ for $\ve x \in \bar P$
 and hence we arrive at an exact algorithm that runs in polynomial time.  
  Note also, that this implies that the precision  $\mu/R$ distinguishes function values of $g(\ve x)$ for $\ve x \in \bar P$.

Lastly, we show that we can solve $\min\{-g(\ve x) : \ve x \in \bar P \cap \Z^n\}$ for any $\bar P \in \mathcal P$.
Since $g$ is convex, $-g$ is concave and hence is also quasi-concave.   By Theorem~\ref{thm:convex-concave} (\ref{itm:quasi-concave}), the minimization problem can be solved exactly provided that we have a comparison oracle for $-g$.  Our comparison oracle is realized by approximating $-g$ to a precision of $\mu/R$, which, as stated above, distinguishes function values of $g$.
This precision can be obtained by approximating the square roots in $g$ using, for instance binary search.  See also~\cite{mansour_complexity_1989} for approximation of square roots.

Thus, we have shown that the set $\Omega_{\mathcal P}$ is indeed a class of functions that we can minimize in polynomial time over polytopes $\bar P \in \mathcal P$.
\end{proof}

\begin{Remark}
In the proof of Theorem~\ref{thm:main}, there is an alternative method that avoids computing solutions to $\min \{ -g(\ve x) \ :\ x \in \bar P \cap \Z^n \}$,  thus avoiding potential numerical issues when approximating square roots.
In fact, let $\ve{\bar x} \in \arg \min \{ f(\ve x) : \ve x \textrm{ is a vertex of } (\bar P \cap \Z^n)_I \}$ and $\ve{\tilde x} \in \arg \min \{ -g(\ve x) \ :\ x \in \bar P \cap \Z^n \}$. Since $-g$ is concave, its optimal value over $\bar P \cap \Z^n$ is attained at a vertex of $(\bar P \cap \Z^n)_I$. Therefore, because $\ve{\bar x}$ minimizes $f$ over the vertices of $(\bar P \cap \Z^n)_I$, we have $f(\ve{\bar x}) \le f(\ve{\tilde x})$.  
\end{Remark}

Finally, we mention that for certain cases, we can obtain exact solutions to Problem~\eqref{eq:problem-general} for $f(\ve x) = \ve x^T Q \ve x$ where $Q$ has at most one negative or at most one positive eigenvalue.  For this, we use the following lemma.

\begin{lemma}
\label{lem:f_quasiconvex}
Let $S \in \Z^{n \times n}$ be non-singular and $D = \mathrm{diag}(\ve d)$ with $\ve d \in \Z^{n-1}_{\ge 0} \times \Z_{< 0}$.
Then $f(\ve x) := \ve x^T S D S^T \ve x$ is quasi-convex on the sets $C_1 := \{ \ve x \in \R^n \,:\, f(\ve x) \le 0,\ S_n \cdot \ve x \ge 0 \}$ and $C_2 := \{ \ve x \in \R^n \,:\, f(\ve x) \le 0,\ S_n \cdot \ve x \le 0 \}$.
\end{lemma}
\begin{proof}
We show this for the set $C_1$; the proof for the set $C_2$ is similar.
We only need to consider the case where $S$ is the identity matrix, because convexity is preserved under linear transformations.

$C_1 = \{ \ve x \in \R^n \,:\, \sqrt{\sum_{i=1}^{n-1} d_i (x_i)^2} \le \sqrt{- d_n} x_n,\ x_n \ge 0 \}$ is a convex because it is a second-order cone~\cite{Boyd2004}. In particular, its interior, which is given by $\intr(C_1) = \{ \ve x \in \R^n \,:\, f(\ve x) < 0,\ x_n \ge 0 \}$, is convex.
 
Note that if $d_i = 0$, then $f$ is quasi-convex on some set $C$ if and only if the restriction of $f$ to $x_i = 0$ is quasi-convex on $C \cap \{\ve x \in \R^n: x_i = 0 \}$. Therefore, by reduction of variables, assume without loss of generality that $d_i \ne 0$ for $i = 1, \ldots, n$.

Since $d_i \ne 0$ for $i = 1, \ldots, n$, $C_1$ is a pointed cone and therefore does not contain any lines.
Let $\ve x, \ve y \in \intr(C_1)$ and $\lambda \in (0, 1)$. We will show that $f(\lambda \ve x + (1-\lambda) \ve y) \le \max \{ f(\ve x), f(\ve y)\}$, thus proving quasi-convexity of $f$ on $\intr(C_1)$.
Let
\begin{align*}
g(\lambda)
&:= f(\lambda \ve x + (1-\lambda) \ve y)
\\
&\phantom{:}= \lambda^2 f(\ve x) + (1- \lambda)^2 f(\ve y) \\
&\phantom{:=} + 2 \lambda (1-\lambda) \sum_{i = 1}^n d_i (S_i \cdot \ve x) (S_i \cdot \ve y) ,
\end{align*}
and note that $g$ is a quadratic function in $\lambda$ and satisfies $g(0) = f(\ve x)$ and $g(1) = f(\ve y)$.
If $g$ is convex, then the result follows by convexity.
Otherwise, $g$ is strictly concave.
Assume without loss of generality that $f(\ve x) \ge f(\ve y)$.
We will show that $g$ is increasing on the interval $[0, 1]$, thus proving the result.

Assume for the sake of contradiction that $g$ is not increasing on the interval $[0, 1]$ and therefore has a local maximum on this interval, say for $\hat \lambda$. Since $g$ is quadratic and strictly concave, $\hat \lambda$ is also its global maximum. By convexity of $\intr(C_1)$, $\hat \lambda \ve x + (1 - \hat \lambda) \ve y \in \intr(C_1)$ and thus $g(\hat \lambda) < 0$.
However, $\intr(C_1)$ does not contain any line. Therefore, there is a $\bar \lambda \in \R$ such that $\bar \lambda \ve x + (1-\bar \lambda) \ve y \not \in \intr(C_1)$, implying $f(\bar \lambda \ve x + (1-\bar \lambda) \ve y) \ge 0$ because $f$ takes the value 0 on the boundary of $C_1$. By maximality of $g(\hat \lambda)$, this implies that $g(\hat \lambda) \ge 0$, a contradiction.
Therefore, $f$ is quasi-convex on $\intr(C_1)$.

Finally, if a continuous function is continuous on the closure of a convex set $C$ and quasi-convex on its interior, it is quasi-convex on the closure of $C$ (see, for example, Theorem~2.2.12 in \cite{cambini_generalized_2009}). Therefore $f$ is quasi-convex in $C_1$.
\end{proof}

\begin{proposition}
\label{prop:main}
Let $Q \in \Z^{n \times n}$ be a symmetric matrix and let $n$ be fixed. Then there is a polynomial time algorithm to solve Problem~\eqref{eq:problem-general} with $f(\ve x) = \ve x^T Q \ve x$ in the following cases:
\begin{enumerate}[(i)]
\item $Q$ has exactly one negative eigenvalue and $f(\ve x_{\opt}) \leq 0$;
\label{itm:prop_min}
\item $Q$ has exactly one positive eigenvalue and $f(\ve x_{\opt}) \geq 0$.
\label{itm:prop_max}
\end{enumerate}
Furthermore, in polynomial time, we can detect if one of these cases occurs. 
\end{proposition}
\begin{proof}
We suppose that $P \cap \Z^n \neq \emptyset$.
As mentioned before, we can determine the inertia of $Q$ in polynomial time.

Suppose $Q$ has exactly one negative eigenvalue.  Consider the sets $C_1$ and $C_2$ as in Lemma~\ref{lem:f_quasiconvex}.  Since $f$ is quasi-convex on each of those sets, we can minimize $f$ over $C_1 \cap P \cap \Z^n$ and $C_2 \cap P \cap \Z^n$ in polynomial time by  Theorem~\ref{thm:convex-concave} (\ref{itm:quasi-convex-semi-algebraic}).  As $C_1 \cup C_2 = \{ \ve x \in \R^n : f(\ve x) \leq 0\}$, we have shown how to detect if $f(\ve x_{\opt}) \leq 0$ and if so, determine an optimal solution.

Suppose instead that $Q$ has exactly one positive eigenvalue.  We consider the decomposition in Lemma~\ref{lem:f_quasiconvex} for $-Q = S (-D) S^T$, which has exactly one negative eigenvalue.
Since we decomposed $-Q$ instead of $Q$, the sets $C_1$, $C_2$ are 
$C_1 := \{ \ve x \in \R^n \,:\, f(\ve x) \geq 0,\ S_n \cdot \ve x \ge 0 \}$ and $C_2 := \{ \ve x \in \R^n \,:\, f(\ve x) \ge 0,\ S_n \cdot \ve x \le 0 \}$.
Now we compute the vertices of the integer hulls  $P^1_I := (P \cap \{ \ve x \in \R^n : S_n \cdot \ve x \geq 0\})_I$ and  $P^2_I := (P \cap \{ \ve x \in \R^n : S_n \cdot \ve x \leq 0\})_I$, which can be done in polynomial time by Theorem~\ref{thm:convex-concave} (\ref{itm:vertices-integer-hull}).
A standard convexity argument shows that $f(\ve x_{\min}) < 0$ if and only if there exists a vertex $\ve x_v$ of $P^1_I$ or $P^2_I$ such that $f(\ve x_v) < 0$.
Otherwise, all vertices of $P^1_I$ and $P^2_I$ are contained in $C_1 \cup C_2$. This certifies that $P\cap \Z^n \subseteq C_1 \cup C_2$, and hence $f(\ve x_{\opt}) \geq 0$.
If this happens, $f$ is quasi-concave on both $C_1$ and $C_2$. By Theorem~\ref{thm:convex-concave} (\ref{itm:quasi-concave}), we can compute an optimal solution in polynomial time.
\end{proof}

\bibliographystyle{abbrv}
\bibliography{references.bib}

\end{document}